\documentclass[mrl, nocrop]{ipart_v1}

\Vol{}
\Issue{}
\Year{}
\firstpage{0}

\usepackage[utf8]{inputenc}
\usepackage[english]{babel}
\usepackage{amsmath, amssymb, amsthm}
\usepackage{mathrsfs}
\usepackage{mathtools}
\usepackage{tikz-cd}
\usepackage{enumerate}
\usepackage{seqsplit}

\newcommand{\hash}[1]{{\ttfamily\seqsplit{#1}}}

\DeclareMathAlphabet{\mathpzc}{OT1}{pzc}{m}{it}

\providecommand{\BG}{{\mathbb{G}}}
\providecommand{\BN}{{\mathbb{N}}}
\providecommand{\BP}{{\mathbb{P}}}
\providecommand{\BQ}{{\mathbb{Q}}}
\providecommand{\BZ}{{\mathbb{Z}}}
\providecommand{\CF}{{\mathpzc{f}}}
\providecommand{\CD}{{\mathcal{D}}}
\providecommand{\CE}{{\mathcal{E}}}
\providecommand{\CL}{{\mathcal{L}}}
\providecommand{\CM}{{\mathcal{M}}}
\providecommand{\CO}{{\mathcal{O}}}
\providecommand{\CR}{{\mathcal{R}}}
\providecommand{\CT}{{\mathcal{T}}}
\providecommand{\CU}{{\mathcal{U}}}
\providecommand{\CV}{{\mathcal{V}}}
\providecommand{\CX}{{\mathcal{X}}}
\providecommand{\CY}{{\mathcal{Y}}}
\providecommand{\CZ}{{\mathcal{Z}}}

\newcommand{\ov}[1]{\overline{#1}}
\newcommand{\wt}[1]{\widetilde{#1}}

\DeclareMathOperator{\Spec}{Spec}
\DeclareMathOperator{\Gal}{Gal}
\DeclareMathOperator{\Frob}{Frob}
\DeclareMathOperator{\Irr}{Irr}

\DeclareMathOperator{\maxdeg}{maxdeg}
\DeclareMathOperator{\N}{N}

\newtheorem{theorem}{Theorem}[section]
\newtheorem{proposition}[theorem]{Proposition}
\newtheorem{lemma}[theorem]{Lemma}
\newtheorem{corollary}[theorem]{Corollary}
\newtheorem{conjecture}[theorem]{Conjecture}
\theoremstyle{definition}
\newtheorem{definition}[theorem]{Definition}
\theoremstyle{remark}
\newtheorem{example}[theorem]{Example}
\newtheorem{remark}[theorem]{Remark}

\begin{document}

\title{Arithmetic Surjectivity for Zero-Cycles}
\author{Damián Gvirtz}

\begin{abstract}
Let $f:X\to Y$ be a proper, dominant morphism of smooth varieties over a number field $k$. When is it true that for almost all places $v$ of $k$, the fibre $X_P$ over any point $P\in Y(k_v)$ contains a zero-cycle of degree $1$? We develop a necessary and sufficient condition to answer this question.

The proof extends logarithmic geometry tools that have recently been developed by Denef and Loughran-Skorobogatov-Smeets to deal with analogous Ax-Kochen type statements for rational points.
\end{abstract}

\maketitle

\section{Introduction}
In \cite{pseudo-split}, Loughran-Skorobogatov-Smeets develop, building upon work of Denef \cite{denef}, a necessary and sufficient criterion to say when a morphism of varieties over a number field $k$ is surjective on $k_v$-points for almost all finite places $v$. This property is called \emph{arithmetic surjectivity} by Colliot-Thélène \cite[\S13]{cime}. More precisely, Loughran et.\ al.\ define that a variety $X$ over a perfect field is \emph{pseudo-split} if every Galois automorphism over the ground field fixes some geometric irreducible component of $X$ of multiplicity $1$. They then prove:

\begin{theorem}\cite[Theorem 1.4]{pseudo-split}\label{thm:lss}
 Let $f:X\to Y$ be a dominant morphism between proper, smooth, geometrically integral varieties over a number field $k$ with geometrically integral generic fibre.
 
 Then $f$ is arithmetically surjective if and only if for each modification $f':X'\to Y'$ of $f$ and for each codimension $1$ point $\vartheta'$ in $Y'$, the fibre $f'^{-1}(\vartheta')$ is pseudo-split.
\end{theorem}

By a \emph{modification} of $f$, we mean a morphism $f':X'\to Y'$ of proper, smooth, geometrically integral varieties over $k$ such that there exist proper, birational morphisms $\alpha_X:X'\to X$  and $\alpha_Y:Y'\to Y$ with $f'\circ\alpha_X=\alpha_Y\circ f$.

In this paper, we closely follow and extend the methods from \cite{pseudo-split} to deal with the analogous question for zero-cycles. We introduce the notion of \emph{combinatorial cycle-splitness} and prove:
\begin{theorem}\label{thm:main}
 Let $f:X\to Y$ be a dominant morphism between proper, smooth, geometrically integral varieties over a number field $k$ with geometrically integral generic fibre.
 
 The following statements are equivalent:
 \begin{enumerate}[(i)]
  \item For almost all places $v$, $f$ has a $v$-adic zero-cycle of degree $1$ in all fibres over $k_v$-points.
  \item For each modification $f':X'\to Y'$ and for each codimension $1$ point $\vartheta'$ in $Y'$, the fibre $f'^{-1}(\vartheta')$ is combinatorially cycle-split.
 \end{enumerate}

\end{theorem}

A situation where Theorem\nobreakspace \ref {thm:main} applies but not Theorem\nobreakspace \ref {thm:lss} is given at the end of this article in Example\nobreakspace \ref {ex:cycle-surjective}.

Note that we do not naively ask for surjectivity on zero-cycles but only for zero-cycles that are each entirely contained in a fibre. This has three reasons. First, if we allowed for zero-cycles whose summands lie in several distinct fibres, the question would not be fibre-wise anymore and our tools would not suffice to provide an answer for $\dim Y>1$. Secondly, the naive version is not very well-behaved even in dimensions $0$ and $1$, which we can handle, where it already leads to rather complicated criteria.

Thirdly, it can be argued that the problem as posed above arises more naturally, for example when considering Artin's conjecture on $p$-adic forms in its variant for zero-cycles of degree $1$.
\begin{conjecture}[e.g.\ {\cite[Problem 3]{kato-kuzumaki}}]\label{conj:artin}
 If $p$ is an arbitrary prime and if $n$ and $d$ are positive integers such that $n\geq d^2$, then a degree $d$ hypersurface in $\BP^n_{\BQ_p}$ has a zero-cycle of degree $1$.
\end{conjecture}
In other words, this open conjecture posits that the famous Ax-Kochen theorem, a special application of Theorem\nobreakspace \ref {thm:lss}, holds without the need to exclude any primes when restated for zero-cycles of degree $1$. In moduli terms, this asks for fibre-wise $p$-adic zero-cycles of degree $1$ in the universal family of such hypersurfaces for every prime $p$.

\subsection{Notation and conventions} By a variety, we mean a separated scheme of finite type over a field $K$. We denote by $\ov X$ the base change of a variety $X$ along an algebraic closure $\ov K$ of $K$. For a field $K'\supset K$, we write $X_{K'}$ for $X\times_K K'$. If $k$ is a number field and $S$ a finite set of finite places in $k$, we write $\CO_k$ for the ring of integers of $k$ and $\CO_{k,S}$ for the $S$-integers of $k$. Furthermore, for a finite place $v$ of $k$, $k_v$ shall denote the completion at $v$ with ring of integers $\CO_{k_v}$ and residue field $k(v)$ of size $N(v)$.

By a model of a variety $X$ over $k$ (respectively $k_v$), we mean a scheme $\CX$ which is flat and of finite type over $\CO_{k,S}$ for some finite set of places $S$ (respectively $\CO_{k_v}$) together with an isomorphism of its generic fibre to $X$. If $X$ is proper, $\CX$ a fixed model of $X$ and $x\in X$ is a closed point, we write $\wt x$ for the closure of $x$ in $\CX$. By a model of a morphism of varieties $f:X\to Y$ over $k$ (respectively $k_v$), we mean a morphism $\CF:\CX\to\CY$ over $\CO_{k,S}$ (respectively $\CO_{k_v}$) such that $\CX$ and $\CY$ are models of $X$ and $Y$ compatible with $f$ in the obvious way.

\section{Preliminary definitions}\label{sec:term}
To start, we introduce some terminology related to zero-cycles and our question.
\begin{definition}
 A variety over a field $K$ is \emph{$r$-cycle-split}, if it contains a zero-cycle of degree $r$ which is the sum of smooth points.
 
 A variety over a number field $k$ is \emph{locally $r$-cycle-split} outside a finite set of places $S$, if for all finite places $v\notin S$ of $k$, the base change $X_{k_v}$ is $r$-cycle-split.
\end{definition}

\begin{definition}\label{def:cycle-surjective}
 A morphism of varieties over a field $K$ is \emph{$r$-cycle-surjective}, if the fibre over any rational point contains a zero-cycle of degree $r$.
 
 A morphism of  varieties over a number field $k$ is \emph{arithmetically $r$-cycle-surjective} outside a finite set of places $S$, if for all finite places $v\notin S$ of $k$, the base change $f\times_k{k_v}$ is $r$-cycle-surjective.
\end{definition}

In the case $r=1$, we propose the easier terminology \emph{cycle-split}, \emph{locally cycle-split}, \emph{cycle-surjective} and \emph{arithmetically cycle-surjective}. Although Theorem\nobreakspace \ref {thm:main} is only concerned with arithmetic cycle-surjectivity, dealing with the case of general $r$ does not add further complications. The terms are chosen in relation to \cite{pseudo-split}.

It turns out to be important to bound the degree of points appearing in zero-cycles.
\begin{definition}
 For a zero-cycle $Z=\sum n_i x_i$ on a variety over a field $K$, define the \emph{maximum degree} of $Z$
 \[\maxdeg Z=\max [K(x_i):K],\]
 where $K(x_i)$ is the residue field of the point $x_i$.
\end{definition}
We will make crucial use of a uniform version of the Lang-Weil estimates \cite{lang-weil}.
\begin{lemma}\label{lem:lang-weil}
 There exists a function $\ov \Phi:\BN^3\to\BN$ with the following property. Let $U\subset \BP^\nu$ be a geometrically irreducible, quasi-projective variety over a finite field, $\ov U$ its closure in $\BP^\nu$ and $\partial U=\ov U\setminus U$.
 
 Then there exists a zero-cycle $Z$ of degree $1$ on $U$ with \[\maxdeg Z\leq \ov\Phi(N,\deg \ov U, \deg\partial U).\]
\end{lemma}
If $X$ is proper and $\iota:X\dashrightarrow\BP^\nu$ a rational embedding defined on an open $U\subset X$, then we will write $\ov\Phi(\iota)$ for $\ov\Phi(N,\deg \ov{\iota(U)}, \deg\partial(\iota(U)))$.

\section{Combinatorial cycle-splitness}
We define the notion of combinatorial cycle-splitness, first for algebras and then for varieties.

\subsection{In dimension \texorpdfstring{$0$}{0}}\label{sub:base0}
Let $X$ be a finite étale scheme scheme over a field $K$. It can be written as $X=\Spec(A)$ for some finite $K$-algebra $A=\oplus_{i=1}^n K_i$ (where $K_i/K$ are finite field extensions but not necessarily normal). Let the Galois extension $L/K$ be the compositum of the Galois closures of the $K_i$ and denote $G:=\Gal(L/K)$.

Let $H_i:=\Gal(L/K_i)$, i.e.\ $L^{H_i}=K_i$. We note that $X$ has a global zero-cycle of degree $r$, if and only if $\gcd_i(\#G/\#H_i)|r$. An element $g\in\Gal(L/K)$ acts on the set $G/H_i$ of right cosets from the right and partitions it into $r_i$ orbits of sizes which we denote by $m_{i1}^g,\dots,m_{ir_i}^g$.

\begin{definition}
 Define the \emph{combinatorial index of $X$ at $g\in G$} as \[I_X(g):=\gcd_{i,j}(m_{ij}^g).\] We call $X$ \emph{combinatorially $r$-cycle-split} if and only if $I_X(g)|r$ for all $g\in G$. If $r=1$, we say $X$ is \emph{combinatorially cycle-split}.
\end{definition}

For the rest of this section, we take $K$ to be a number field $k$. With notation as above, the extension $L/k$ is unramified outside a finite set of places $S$. A finite place of $k$ that is unramified in all $K_i$ is also unramified in $L$. For a finite place $v\notin S$, let $\Frob_v\in G$ be the Frobenius automorphism at $v$.

\begin{lemma}
 Let $v\notin S$ be a finite place of $k$.  Then
 \[A\otimes k_v=\bigoplus_{i=1}^n\bigoplus_{j=1}^{r_i}k_{ij}\]
 where $k_{ij}/k_v$ is a finite extension of degree $m_{ij}^{\Frob_v}$.
\end{lemma}
\begin{proof}
 This is \cite[Theorem 33]{marcus}.
\end{proof}

Note that the list of orbit sizes really only depends on the conjugacy class of $g$: the size of the orbit of $H_it$ under $g$ is the smallest integer $j$ such that $tg^j\in H_it$, or equivalently $g^j\in t^{-1}H_it$.

\begin{corollary}
 Let $X$ be a finite étale scheme over a number field $k$ and $S$ a finite set of places such that $L/k$ is unramified outside $S$. For a finite place $v\notin S$, $X_{k_v}$ is $r$-cycle-split, if and only if $\gcd_{i,j}(m_{ij}^{\Frob_v})|r$.
\end{corollary}

\begin{corollary}
 Let $X$ be a finite étale scheme over a number field $k$ and $S$ a finite set of places such that $L/k$ is unramified outside $S$. Then $X$ is locally cycle-split outside $S$, if and only if $X$ is combinatorially $r$-cycle-split.
\end{corollary}
\begin{proof}
 One direction directly follows from the previous corollary. The other direction follows because by Cebotarev density, for every conjugacy class $C\subseteq G$, there exist infinitely many places $v$ with $\Frob_v\in C$. Hence, if $X$ is not combinatorially $r$-cycle split there exists a $v\notin S$ such that $\gcd_{i,j}(m_{ij}^{\Frob_v})\nmid r$.
\end{proof}

\begin{example}\label{ex:upgrade}
 The preceding corollary gives a very explicit condition that can be explicitly checked for a finite group $G$. One example of an everywhere locally cycle-split scheme that is not cycle-split is \[\Spec(k[t]/(t^2-a)(t^2-b)(t^6-ab))\] with $a,b,a/b\notin k^2$. In the case where $a$ or $b$ is a square in $k_v$, the scheme has a rational point. If $v$ does not lie over $2$, $a,b\in \CO_{k_v}^\times$ and neither $a$ nor $b$ are squares in $k_v$, then $ab$ is a square in $k_v$ and we get $k_v$-points of degree $2$ and $3$, hence a zero-cycle of degree $1$. 
 
 In fact, this is a ``modification'' of an example by Colliot-Thélène for non-split pseudo-splitness where the exponent $6$ is replaced by $2$ \cite[4.1]{squaretrick}. 
\end{example}

\begin{example}\label{ex:non-upgrade}
 Take \[X=\Spec(\BQ[t]/(t^2+1)(t^6-3t^2-1))\] which has a local zero-cycle of degree $1$ everywhere.
 The second factor $(t^6-3t^2-1)$ is an irreducible polynomial that is everywhere reducible. This is because its non-cyclic Galois group is $A_4$, of which a subgroup of order $2$ leaves $\BQ[t]/(t^6-3t^2-1)$ fixed. Moreover, due to the absence of subgroups of order $6$ in $A_4$, locally there always is a factor of order dividing $3$ which together with $(t^2+1)$ yields a zero-cycle of degree $1$. 
 
 Moreover, $X$ is not a finite cover of a non-split pseudo-split scheme $X'$ over $\BQ$ as in Example~\ref{ex:upgrade}. This is because $A_4$ is the smallest counterexample to the converse Lagrange's theorem and thus one sees that any proper quotient of $\BZ/2\times A_4$ fails to satisfy even the group theoretic condition for combinatorial cycle-splitness.
\end{example}

It is a curious result that there is no connected example ($n=1$) as the following theorem shows.
\begin{theorem}
 The Hasse principle for zero-cycles of degree $1$ holds for connected, reduced zero-dimensional schemes over a number field $k$. 
\end{theorem}
\begin{proof}
As before, let $L/k$ be a finite non-trivial Galois extension with Galois group $G$ and $H\subsetneq G$ a proper subgroup. We want to show that $\Spec L^H$ is not locally $r$-cycle-split at infinitely many places. Equivalently, we want to find an element $g$ such that \[\gcd_{t\in G}\min\{k|g^k\in t^{-1}Ht\}>1.\]

To do this we use the following fact proven ``outrageous[ly]'' in \cite[Theorem 1]{fein} via the classification of finite simple groups: for a finite group $G$, there exists a prime number $p$ and an element $g\notin \bigcup_{t\in G}t^{-1}Ht$ of order a power of $p$. This is sufficient since then $p|\min\{k|g^k\in t^{-1}Ht\}$ for all $t\in G$.
\end{proof}
In more down-to-earth language, there is no irreducible polynomial over $k$ that factors into coprime degrees modulo almost all primes.

\subsection{In higher dimensions}\label{sub:dim0}
For the beginning of this section, let us again allow $K$ to be any field. Let $X$ be a proper variety over $K$. For $X'$ a reduced, irreducible component of $X$, we define the \emph{(apparent) multiplicity} of $X'$ in $X$ as the length of the local ring $\CO_{X,\eta'}$ where $\eta'$ is the generic point of $X'$. We define the \emph{geometric multiplicity} of $X'$ in $X$ as the length of the local ring $\CO_{\ov X, \ov{\eta'}}$ where $\ov{\eta'}$ is a point of $\ov X $ lying over $\eta'$. If $X'$ is geometrically reduced, for example when $K$ is perfect, then multiplicity and geometric multiplicity coincide.

Let $X_1^m,\dots,X_n^m$ be the reduced, irreducible components of geometric multiplicity $m$ in $X$. Let $K_i$ be the separable closure of $K$ in the function field of $X_i^m$.

\begin{definition}
 Define the \emph{algebra of irreducible components of geometric multiplicity $m$} as $Z_X^m:=\Spec(\oplus_{i=1}^n K_i)$. (If there are no such components, then $Z_X^m$ is empty.)
\end{definition}

The reason for this definition is of course that the embedding of the ground field into the function field  of a scheme controls, to some extent, its geometric properties and thus we can reduce to the previous section. A scheme $T$ of finite type over $K$ is geometrically irreducible if and only if $T$ is irreducible and $K$ is separably closed in the function field of $T$ (see \cite[4.5.9]{ega4-2}).

However, using the functor of open irreducible components defined by Romagny we can obtain finer results. For a finite type morphism of schemes $T\to R$ with $R$ integral, let $\Irr^m_{T/R}$ be the subfunctor of $\Irr_{T/R}$ defined in \cite[Def. 2.1.1]{romagny} of open irreducible components of geometric multiplicity $m$. We recall that $\Irr_{T/R}$ parametrises open subschemes $U$ of an $R$-scheme $R'$ such that the geometric fibres of $U\times_R R'\to R'$ are interiors of irreducible components in the geometric fibres of $T\times_R R'\to R'$. Note that this is stable under base change and thus functorial because we use the geometric instead of the apparent multiplicity.

\begin{lemma}\label{lem:irr-functor}
  The functor $\Irr^m_{T/R}$ is representable over a dense open of $R$ by a finite étale cover.
\end{lemma}
\begin{proof}
  Let $\eta$ be the generic point of $R$ and $T'\hookrightarrow T\to R$ be the reduced closure of the irreducible components of geometric multiplicity $m$ in the fibre over $\eta$. Then after replacing $R$ with a dense open subscheme, we have that $\Irr^m_{T/R}=\Irr_{T'/R}$ because the geometric multiplicity of the fibre over $\eta$ spreads out to a dense open neighbourhood by \cite[Proposition 9.8.6]{ega4-3}.
  
  After further replacement of $R$ with a dense open subscheme, the functor $\Irr_{T'/R}$ is representable by a separated algebraic space which is finite étale over $R$ by \cite[2.1.2,2.1.3]{romagny}. However, by Knutson's representability criterion, this algebraic space over $R$ must in fact be a scheme (cf.\ \cite[Proof of Proposition 3.7]{smeets-loughran} for this last step).  
\end{proof}

\begin{lemma}\label{lem:irr-rep}
 The functor $\Irr^m_{X/K}$ is represented by $Z_X^m$.
\end{lemma}
\begin{proof}
 This follows from \cite[2.1.4]{romagny}.
\end{proof}

\begin{definition}\label{def:general-index}
 Let $G$ be the Galois group defined in subsection \ref{sub:dim0} for the finite étale $K$-scheme $Z_X^m$. Define the \emph{combinatorial index of $X$ at $g\in G$} as \[I_X(g):=\gcd_m(mI_{Z_X^m}(g)).\] We call $X$ \emph{combinatorially $r$-cycle-split} if and only if $I_X(g)|r$ for all $g\in G$. If $r=1$, we say $X$ is \emph{combinatorially cycle-split}.
\end{definition}
This is compatible with the previous definition of combinatorial index in dimension $0$ and only depends on the conjugacy class of $g$ in $G$.

Let us return to the case of $k$ a number field and assume $X$ is smooth and proper over $k$. Let $v$ be a finite place of $k$. To tackle the question of zero-cycles on $X_{k_v}$, we need to relate closed points in the special and generic fibres of a model. This seems to be folkloric knowledge partly written down in \cite[\S9, Cor. 9.1]{blr} but the author could not find a complete reference before \cite{bosch-liu} (see also \cite[4.6]{wittenberg-c11} and \cite[2.4]{kesteloot}).

\begin{lemma}\label{lem:generalised-hensel}
 Let $\CX$ be a proper, flat model over $\CO_{k_v}$ of $X_{k_v}$. Let $\ov x\in \CX(k(v))$ be a point which is regular in $\CX$ and regular in the reduction $\CX_{k(v)}$ and lies on a geometrically irreducible component of $\CX_{k(v)}$ of multiplicity $m$. Then there exists a closed point $x\in X_{k_v}$ of degree $m$ with reduction $\ov x$.
\end{lemma}
\begin{proof}
 See \cite[2.3]{ct-saito}.
\end{proof}

Conversely, the following result applies.
\begin{lemma}\label{lem:converse-hensel}
 Let $\CX$ be a proper, regular, flat model over $\CO_{k_v}$ of $X_{k_v}$ and $x$ a closed in point in $X_{k_v}$ of degree $d$ with reduction $\tilde x$.
 
 Let $D_j$, $j\in J$, be the irreducible components of $X_{k(v)}$ on which $\tilde x$ lies. Denote by $m_j$ the multiplicity of $D_j$ and by $d_j$ the minimal degree of an extension of $k(v)$ over which  $D_j$ splits into geometrically irreducible components. Then $\gcd_{j\in J} m_j d_j$ divides $d$.
\end{lemma}
\begin{proof}
 See \cite[1.6]{bosch-liu}.
\end{proof}

\begin{lemma}\label{lem:special-generic}
 Let $\CX$ be a proper, normal, flat model over $\CO_{k_v}$ of $X_{k_v}$. Let $\Frob_v$ be the Frobenius element in the absolute Galois group of $k(v)$.
 
 If $I_{\CX_{k(v)}}(\Frob_v)|r$, then $X_{k_v}$ is $r$-cycle-split. If $\CX$ is regular, then the converse holds. 
 
 There exists a function $\Phi:\BN^3\to\BN$ not depending on $X$ with the following property. Let $\iota:\CX_{k(v)}\dashrightarrow\BP^\nu_{k(v)}$ be a rational embedding. If \[I_{\CX_{k(v)}}(\Frob_v)|r,\] then there exists a zero-cycle $Z$ of degree $r$ on $X_{k_v}$ with $\maxdeg Z\leq \Phi(\iota)$ (where $\Phi(\iota)$ is defined as after Lemma~\ref{lem:lang-weil}).
\end{lemma}
\begin{proof}
 Assume $I_{\CX_{k(v)}}(\Frob_v)|r$. Then there exist geometrically irreducible components $D_j$, $j\in J$, of the special fibre of multiplicities $m_j$ defined over extensions of $k(v)$ of degrees $d_j$ s.t. \[\gcd_{j\in J}d_jm_j=I_{\CX_{k(v)}}(\Frob_v).\]
 By the Lang-Weil estimates as formulated in Section\nobreakspace \ref {sec:term}, each $D_j$ has a zero-cycle $Z$ of degree $d_j$. 
 
 Let $Z_j$ be the union of the non-regular locus of $\CX$ and the non-regular locus of the reduction of $D_j$. Because $\CX$ is assumed normal, hence regular in codimension $1$, $Z_j$ does not contain all of $D_j$. Then $\deg Z_j$ has an upper bound only depending on $\nu$ and the degree of the image of $\iota$. By the Lang-Weil estimates as described in Section~\ref{sec:term}, one can arrange for the summands of $Z$ to avoid all $Z_j$ and satisfy $\maxdeg Z\leq\Phi(\iota)$ for a suitable function $\Phi$.
 
 Applying Lemma~\ref{lem:generalised-hensel} to each of the summands, the existence of points of orders $m_jd_j$ and thus a zero-cycle of degree $r$ in $X_{k_v}$ follows.
 
 The converse in the case of regular $\CX$ follows from Lemma~\ref{lem:converse-hensel}.
\end{proof}

\begin{remark}
 We remark that to examine $r$-cycle-splitness of the special fibre itself, all components of multiplicity greater than $1$ would have to be discarded. Thus, there are two notions, $r$-cycle-split and combinatorially $r$-cycle-split. This is a difference to the case of rational points with only one notion of pseudo-split.
\end{remark}

\begin{lemma}\label{lem:variety-cycle-split}
 Let $X$ be a smooth, proper variety over a number field $k$. Let $\iota: X\dashrightarrow\BP^\nu$ be a rational embedding of $X$. Then $X$ is almost everywhere locally $r$-cycle-split if and only if $X$ is combinatorially $r$-cycle-split. In this case, $X_{k_v}$ has a zero-cycle $Z$ of degree $r$ with $\maxdeg Z\leq \Phi(\iota)$ for all $v\notin S$.
\end{lemma}
\begin{proof}
 Let $U\subseteq X$ be a dense open subvariety on which $\iota$ is defined. We can find a finite set $S$ of places such that $U\hookrightarrow X$ and $\iota:U\hookrightarrow\BP^\nu$ spread out to models $\CU\hookrightarrow\CX$ and $\iota_S:\CU\hookrightarrow\BP^\nu_{\CO_{k,S}}$ over $\CO_{k,S}$ where $\CU$ and $\CX$ are smooth over $\CO_{k,S}$.
 
 By Lemma\nobreakspace \ref {lem:irr-rep} and Lemma\nobreakspace \ref {lem:irr-functor}, after possibly enlarging $S$, $\Irr_{\CX/\CO_{k,S}}^m$ is represented by $\Spec(\oplus_{i=1}^n \CO_{K_i,S})$. The result now follows from Lemma\nobreakspace \ref {lem:special-generic}.
\end{proof}

\section{\texorpdfstring{$s^{0}$}{s0}-invariants}
In analogy to the $s$-invariants in \cite{pseudo-split}, we construct \emph{$s^{0}$-invariants} that measure failure of combinatorial cycle-splitness in families. Let $f:X\to Y$ be a morphism of varieties over a number field $k$.

For any (possibly non-closed) point $y\in Y$, set $K:=k(y)$. We get finite étale (possibly empty) $K$-schemes $Z^m_{f^{-1}(y)}=\Irr^m_{f^{-1}(y)/K}$ for all multiplicities $m$. We may pick $L/K$ a minimal Galois extension which splits all $Z^m_{f^{-1}(y)}$ with Galois group $G$. Denote by $k_K$ and $k_L$ the algebraic closures of $k$ in $K$ and $L$. By replacing $k_L$ with its Galois closure and extending $L$, we can assume that $k_L/k$ is Galois. Let $N$ be the subgroup of $G$ acting trivially on $k_L$. Denote by $\Omega_{k_K}$ the set of finite places of $k_K$.

\begin{definition}
 For a finite place $v$ of $k$, define $s^{0,r}_{f,y}(v)$ in the following way:
 \begin{enumerate}[(i)]
  \item as $1$, if $v$ ramifies in $k_L$ or there is no place in $k_K$ of degree $1$ over $v$
  \item otherwise, as
  
  \[\frac{\sum_{\substack{w\in\Omega_{k_K}\\ \N(w)=\N(v)\\w|v}}\#\{g\in G:\Frob_w\equiv g\bmod N, I_{f^{-1}(y)}(g)|r\}}{\#N\#\{w\in\Omega_{k_K}|\N(w)=\N(v),w|v\}}.\]
 \end{enumerate}
\end{definition}
One can see that $s^{0,r}_{f,y}(v)$ is constant on the conjugacy class of $\Frob_v$, i.e.\ that this function is Frobenian in the sense of Serre \cite[\S3.3.3.5]{serre} but this fact will not be directly needed.

Over finite fields, the $s^0$-invariants asymptotically quantify the failure of combinatorial $r$-cycle-splitness.
\begin{proposition}\label{prop:asymptotics}
  Assume $Y$ is integral of dimension $n$ with generic point $\eta$. Let $\CF:\CX\to\CY$ be a model of $f$ over $\CO_{k}$. Then
 \begin{align*}
&\#\{y\in\CY(k(v))|\CF^{-1}(y)\text{ is combinatorially $r$-cycle-split}\}\\=&s^{0,r}_{f,\eta}(v)\#\CY(k(v))+O(\N(v)^{n-1/2})
 \end{align*}
 as $\N(v)\to\infty$, where the asymptotic constant of the $O$-notation only depends on the chosen model.
\end{proposition}
\begin{proof}
 The main idea after \cite[Proposition 3.13]{pseudo-split} is to count and then compare both sides using Lang-Weil estimates and the Cebotarev density theorem for schemes. We divide the proof into several parts.
 
 \textbf{Set-up} By the Lang-Weil estimates we can remove strict closed subsets of $\CY$ since for dimension reasons, their rational points only contribute to the error term. Hence, with the help of Lemma\nobreakspace \ref {lem:irr-functor}, we assume that $\Irr^m_{\CX/\CY}\to\CY$ finite étale.
 
 In the same way, we ensure that $\CY$ and its special fibres $\CY_{k(v)}$ are normal for all $v$ not contained in some finite set $S$. Set $y:=\eta$  and from there on $K$, $L$, $k_K$ and $k_L$, $G$ and $N$ as before. Enlarging $S$ further, we may spread out and assume that $L$ is the generic fibre of a Galois closure $\CL$ of $\Irr^m_{\CX/\CY}\to\CY$. From now on, let $v\notin S$
 
 \textbf{Counting points of $\CY_{k_v}$ with Lang-Weil} The functor $\Irr_{\CY_{k(v)}/k(v)}$ is represented by \[\Spec\CO_{k_K}\otimes_{\CO_k} k(v)=\bigoplus_{\substack{w\in\Omega_{k_K}\\w|v}} k(w).\]
 Therefore, geometrically irreducible components of $\CY_{k(v)}$ correspond to places $w$ with $\N(w)=\N(v)$. We write $\CY_w$ for such a component. By the normality assumption, the irreducible components of $\CY_{k(v)}$ are all disjoint, so if there is none which is geometrically irreducible, $\CY_{k(v)}$ has no rational point. This is the trivial case of the proposition. In the non-trivial case, we can count points by Lang-Weil:
 \begin{eqnarray*}\#\CY(k(v))&=&\sum_{\substack{w\in\Omega_{k_K}\\ \N(w)=\N(v)\\w|v}}\#\CY_w(k(w))=\sum_{\substack{w\in\Omega_{k_K}\\ \N(w)=\N(v)\\w|v}}\N(v)^n+O(\N(v)^{n-1/2})\\&=&{\#\{w\in\Omega_{k_K}|\N(w)=\N(v),w|v\}}\N(v)^n+O(\N(v)^{n-1/2}).\end{eqnarray*}
 
 \textbf{Counting combinatorially $r$-cycle-split fibres with Cebotarev} For a rational point $y\in\CY(k(v))$, we can view the Frobenius $\Frob_y$ as an element of $G$ up to conjugacy. The fibre $\CF^{-1}(y)$ is combinatorially $r$-cycle-split if and only if $I_{\CF^{-1}(\eta)}(\Frob_y)=I_{\CF^{-1}(y)}(\Frob_y)|r$. Let $\delta_\CF(g)\in\{0,1\}$ be the indicator function of the set of elements $g\in G$ for which $I_{\CF^{-1}(y)}(g)|r$. This function only depends on the conjugacy class of $g$. Applying the Cebotarev density theorem for étale morphisms as in \cite[9.15]{serre} to $\delta_\CF$ one gets:
 \begin{align*}
  \#\{y\in\CY(k(v))|\CF^{-1}(y)\text{ is combinatorially $r$-cycle-split}\}\\=\frac{\N(v)^n}{\#N}\sum_{\substack{w\in\Omega_{k_K}\\ \N(w)=\N(v)\\w|v}}\#\{g\in G:\Frob_w\equiv g\bmod N, I_{f^{-1}(\eta)}(g)|r\}+O((\N(v))^{n-1/2})
 \end{align*}
 
 Comparing both counts with the definition of $s^{0,r}_{f,\eta}(v)$, the result follows.
\end{proof}

The asymptotic formula gives a necessary condition for combinatorial cycle-splitness of all fibres.
\begin{corollary}\label{cor:combinatorial-fail}
 With the same notation, if $s^{0,r}_{f,\eta}(v)<1$ for some $v\notin S$, then there exists $y\in\CY(k(v))$ such that $\CF^{-1}(y)$ is not combinatorially $r$-cycle-split.
\end{corollary}
\begin{proof}
 For $v$ large enough, there will be rational points on $\CY_{k(v)}$ but by Proposition\nobreakspace \ref {prop:asymptotics}, not all fibres over them can be combinatorially $r$-cycle-split.
\end{proof}

The asymptotics also give the other direction.
\begin{corollary}\label{cor:cycle-split}
 With the same notation, assume $\CY$ is integral normal and $\Irr_\CF^m$ finite étale over $\CY$ for all $m$. Then there exists a finite set of places $S$ such that for all $v\notin S$, $s^{0,r}_{f,\eta}(v)=1$ if and only if the fibre of $\CF$ over every $y\in\CY(k(v))$ is combinatorially $r$-cycle-split.
\end{corollary}
\begin{proof}
 One direction has just been proven. For the other direction, we use the same notation as in Proposition\nobreakspace \ref {prop:asymptotics}.
 
 A point $y\in\CY(k(v))$ must lie on a geometrically irreducible component corresponding to the degree $1$ place $w$ of $k_K$. Let $l\in\CL$ be a closed point over $y$ and $u$ be the corresponding place of its irreducible component. Then
 \[k(y)=k(w)\subset k(u)\subset k(l)\]
 and there exist natural embeddings \[\Gal(k(l)/k(y))\hookrightarrow G\] and \[\Gal(k(u)/k(y))\hookrightarrow G/N.\]
 
 By functoriality of Frobenius, we have \[\Frob_{l/y}\bmod N=\Frob_{u/w}.\] Because of the assumption that $s^{0,r}_{f,\eta}(v)=1$, we deduce that $\Frob_{l/y}$ acts on $\Irr^m_{\CF^{-1}(y)/y}$ such that $I_{\CF^{-1}(y)}(\Frob_{l/y})|r$. Hence $\CF^{-1}(y)$ is combinatorially $r$-cycle-split.
\end{proof}

\begin{corollary}
  The fibre $f^{-1}(y)$ is combinatorially $r$-cycle-split if and only if $s^{0,r}_{f,y}(v)=1$ for almost all $v$.
\end{corollary}
\begin{proof}
 This is Corollary~\ref{cor:cycle-split} in the case of a zero-dimensional base. 
\end{proof}

\section{Arithmetic cycle-surjectivity}
Let $f:X\to Y$ be a dominant morphism between proper, smooth, geometrically integral varieties with geometrically integral generic fibre over a number field $k$. 

\subsection{Birational invariance}
We want to prove that arithmetic $r$-cycle-surjectivity is a property invariant under modifications. The argument here is more subtle than in the case of rational points.

\begin{definition}
Let $v$ be a place of $k$. If a fibre over a $k_v$-point $y$ of $Y$ contains a zero-cycle of degree $r$ we call this cycle a \emph{witness for $r$-cycle-surjectivity over $y$ at $v$}.
\end{definition}

\begin{lemma}\label{lem:birational}
 Let $v$ be a place of $k$. Let $V$ be a dense open subset of $Y$. Assume that there exists $B\in\BN$  such that $f^{-1}(V)\to V$ is $r$-cycle-surjective at $v$ and there exist witnesses $Z_v$ for $r$-cycle-surjectivity over $y$ at $v$ for all $y\in V(k_v)$ with $\maxdeg Z_v\leq B$. Then $f$ is $r$-cycle-surjective at $v$.
\end{lemma}
\begin{proof}
 Assume cycle-surjectivity on an open $V$ with a uniform bound $B$ as described above. Let $k_v(i)$ denote the compositum of all degree $i$ extensions of $k_v$. Then $X(k_v(i))\subseteq X(\overline{k_v})$ is the set of $\overline{k_v}$-points fixed by all elements in $\Gal(\overline{k_v}/k_v(i))$ and this a closed subset. Hence 
 \[Y_B(f):=\bigcup_{\substack{I\subset \{1,\dots, B\}\\ \gcd(I)|r}}\bigcap_{i\in I}f(X(k_v(i)))\]
 is a finite union of closed subsets of $Y(\overline{k_v})$.
 
 Let $y$ be a $k_v$-rational point in $V$ for which the fibre $f^{-1}(y)$ contains a zero-cycle $Z$ of degree $r$ with $\maxdeg Z\leq B$. There exists $I\subset \{1,\dots, B\}$ such that \[y\in\bigcap_{i\in I}f(X(k_v(i)))\subset Y_B(f).\]
 
 On the other hand, a point $y\in Y_B(f)$ lies in $\bigcap_{i\in I}f(X(k_v(i)))$ for some $I\subset \{1,\dots, B\}$ with $\gcd(I)|r$, so its fibre has a closed $k_v(i)$-point for all $i\in I$. Let $j_i$ be the degree of this point. It follows that the prime factors of $j_i$ are contained in the prime factors of $i$. In particular, the fibre has a zero-cycle of degree $\gcd_{i\in I}j_i=\gcd I|r$.
 
 Now $Y_B(f)$ is closed and contains $V(k_v)$ which is dense and open in $Y(k_v)$, hence $Y(k_v)\subseteq Y_B(f)$.
\end{proof}

\begin{remark}
 The above proof generalises to $k_v$ any Henselian (non-trivially) valued field.
\end{remark}

\begin{lemma}
 To show arithmetic $r$-cycle-surjectivity of $f$, it is enough to show arithmetic $r$-cycle-surjectivity of $f^{-1}(V)\to V$ for a dense open $V$ in $Y$.
\end{lemma}
\begin{proof}
 By Lemma\nobreakspace \ref {lem:birational} all we have to show is that for $v$ large enough, if $f$ is arithmetically $r$-cycle-surjective over $V$, there is a uniform bound on the maximum degree of witnesses. By generic smoothness \cite[Corollary 10.7]{hartshorne}, after shrinking $V$, we may assume that all fibres over $V$ are smooth.
 
 Let $\iota: X\dashrightarrow\BP^\nu_Y$ be a rational embedding. But now by Lemma\nobreakspace \ref {lem:variety-cycle-split} (which has a smoothness assumption), a fibre over a point in $V$ is almost everywhere locally $r$-cycle-split if and only if it is almost everywhere locally $r$-cycle-split with zero-cycles as witnesses that have maximum degree less than $\Phi(\iota)$.
\end{proof}

\subsection{Necessary condition}
From the results over finite fields, we can deduce a necessary condition for arithmetic $r$-cycle-surjec\-tivity.

\begin{proposition}\label{prop:transversal}
 Let $\CF:\CX\to\CY$ be a proper model of $f$ over $\CO_{k,S}$ for a finite set of places $S$ of $k$ with regular source and target. Let $\CT\subset\CY$ be a reduced divisor such that $\CF$ is smooth away from $\CT$. Then after possibly enlarging $S$, we can find a subset $\CR\subset\CT_{\CO_{k,S}}$ of codimension at least $2$ in $\CY_{\CO_{k,S}}$ such that for all $v\notin S$ the following holds.
 
 Choose $\wt y\in \CY(\CO_{k_v})$; denote its generic point by $y\in Y(k_v)$. If $\tilde y$ intersects $\CT_{\CO_{k,S}}$ transversally outside $\CR_{\CO_{k,S}}$ and the fibre at $(\tilde y\bmod \pi_v)$ is not combinatorially $r$-cycle-split, then $f^{-1}(y)$ is not $r$-cycle-split. 
\end{proposition}
\begin{proof}
 This is a variant of \cite[Theorem 2.8]{smeets-loughran}: After possibly enlarging $S$, $\CR$ can be chosen of codimension $2$ in a way such that
 \begin{enumerate}
  \item by generic flatness for regular schemes, $\CF$ is flat on the complement $\CY\setminus\CR$,  and
  \item by generic submersivity \cite[Theorem 2.4]{smeets-loughran} in characteristic $0$, $\CF$ is submersive (i.e.\ surjective on tangent spaces) over $\CT\setminus\CR$.
 \end{enumerate}
 Then $\CX\times_\CY \wt y$ is regular and its special fibre is not combinatorially $r$-cycle-split. The rest follows by Lemma\nobreakspace \ref{lem:special-generic}.
\end{proof}

\begin{proposition}\label{prop:non-surjectivity}
 Let $\vartheta\in Y^{(1)}$ be a codimension $1$ point of $Y$. There exists a finite set of places $S$ such that for all $v\notin S$ the following holds: if $s^{0,r}_{f,\vartheta}(v)<1$, then $f$ is not arithmetically $r$-cycle-surjective.
\end{proposition}
\begin{proof}
 If $s_{f,\vartheta}^{0,r}(v)<1$, let $\CE$ be the closure of $\vartheta$ in $\CY$. By Corollary\nobreakspace \ref {cor:combinatorial-fail}, for suitable $S$ we can find a point $y$ in the special fibre of $\CE$ above which the fibre is not combinatorially $r$-cycle-split. By Proposition\nobreakspace \ref {prop:transversal}, it therefore suffices to lift $y$ to an integral point intersecting $\CE$ transversally. The argument for this is well-known and literally the same as in \cite[Theorem 4.2]{pseudo-split} via blowing-up $\CY$ in $y$ and choosing a point on the exceptional divisor. 
\end{proof}

\subsection{Sufficient condition and proof of main theorem}\label{sec:sufficient}
Finally, using tools from logarithmic geometry, we can give a necessary and sufficient criterion for arithmetic $r$-cycle-surjectivity. We refrain from giving yet another exposition of logarithmic geometry and refer the reader to \cite{handbook}. All log schemes in this section will be fs Zariski log schemes.

For this section assume that we have a log smooth, proper model $\CF:(\CX,\CD)\to(\CY,\CE)$ of $f$ where $(\CX,\CD)$ and $(\CY,\CE)$ are Zariski log regular schemes (with divisorial log structure induced by $\CD$ and $\CE$) that are log smooth and proper over $\CO_{k,S}$ equipped with the trivial log structure for some finite set of places $S$. This can be achieved after a modification of $f$ by using Abramovich-Denef-Karu's toroidalisation theorem in \cite{karu} and spreading out. Denote by $D$ and $E$ the generic fibres of $\CD$ and $\CE$. Set $\CU:=\CX\setminus\CD$, $U:=X\setminus D$, $\CV:=\CY\setminus\CE$, and $V:=Y\setminus E$. On these open sets, the log structures are trivial.

By possibly enlarging $S$ in the spreading-out procedure above, we may assume that all irreducible components $\CE'$ of $\CE$ intersect the generic fibre non-trivially, i.e.\ their generic points lie in $Y$. This property of our chosen model is absolutely crucial for the method presented here. Namely, one can control the splitting behaviour of the fibre of $\CF$ over a point in the interior of $\CE'$ by the behaviour of the fibre of $f$ over the generic (characteristic $0$) point $\vartheta'$ of $\CE'$ (see Lemma~\ref{lem:irr-rep-strata}).

Let $v$ be a finite place of $k$. Let $k'/k$ be a finite extension and $w$ an extension of $v$ to $k'$. By the valuative criterion of properness, any closed point \[y:\Spec k'_w\to Y\] extends to a morphism \[\wt y:(\Spec\CO_{k'_w})^\dagger\to (\CY,\CE),\] where $(\Spec\CO_{k'_w})^\dagger$ is the log scheme equipped with the standard divisorial log structure defined by a uniformiser $\pi_w$ (i.e.\ with monoid given by $\CO_{k'_w}\setminus 0$). 

In \cite{kato}, Kato defines the fan $(F_T,\CM_{F_T})$, a locally monoidal space associated to a log regular scheme $T$, and a morphism $c_T:T\to F_T$. The preimage $U(t)$ of a point $t\in F_T$ under $c_T$ is called a \emph{logarithmic stratum} and is a locally closed subset of $T$. Then the points of $F_T$ can be identified with the generic points of the logarithmic strata. (In the older language of toroidal embeddings, these strata are connected components of repeated intersections of the boundary divisor.) To each $t\in F_T$, there corresponds a Kato subcone $F_T^t$ of $F_T$ which is the unique subcone with closed point $t$.

There is an attached \emph{logarithmic height} function $h_T:F_T(\Spec \BN)\to\BN$ which is defined in \cite[\S5]{pseudo-split} as follows. Under the $\BN$-valued point $t\in F_T(\Spec \BN)$, the closed point of $\Spec\BN$ is sent to the closed point of a Kato subcone $\Spec \BN^j$. The height $h_T(t)$ is defined as the sum of the images of the generators of $\BN^j$ under the map $\BN^j\to\BN$ induced by $t$.

Furthermore, a morphism $g$ of log regular schemes induces a morphism $F(g)$ of Kato fans. Because $F_{(\Spec\CO_{k'_w})^\dagger}\cong\Spec\BN$, this defines a logarithmic height $h_\CY(y)$ for any $y\in Y(k_w')$. Morally, the height of $y$ quantifies how often $\wt y$ intersects the special fibre.

\begin{lemma}\label{lem:irr-rep-strata}
 For any $t\in F_\CY$ and $m\in\BN$, the functor $\Irr^m_{\CF^{-1}(U(t))/U(t)}$ is representable by a finite étale scheme over $U(t)$.
\end{lemma}
\begin{proof}
 It is shown in \cite[Proposition 5.18]{pseudo-split} that $\Irr_{\CF^{-1}(U(t))/U(t)}$ is representable by a finite étale scheme over $U(t)$. By \cite[Proposition 5.16]{pseudo-split}, apparent multiplicity is constant along logarithmic strata for proper, log smooth morphisms of log regular schemes, and because log smoothness is stable under base change, the same is true for geometric multiplicity. Thus the subfunctor $\Irr_{\CF^{-1}(U(t))/U(t)}^m$ is represented by the closure of $\Irr_{\CF^{-1}(t))/t}^m$ in $\Irr_{\CF^{-1}(U(t))/U(t)}$.
\end{proof}

The following two propositions bound the intersection behaviour of points in $Y$ the fibres above which we have to consider.
\begin{proposition}\label{prop:intersection-bound}
 There is a positive integer $N$ with the following property. Let $B\in\BN$ be arbitrary and $v\notin S$ a place of $k$. If the fibre over each point $y\in V(k_v)$ with $h_\CY(y)\leq N$ has a zero-cycle $Z$ of degree $r$ with $\maxdeg Z\leq B$, then $f\times_k{k_v}$ is $r$-cycle-surjective.
\end{proposition}
\begin{proof}
 The proof is very similar to the one in \cite[Proposition 6.1]{pseudo-split}, which itself is an adaptation of \cite[4.2]{denef}, and we only sketch the steps and highlight the necessary changes.
 
 Let $F(f)_*:F_X(\Spec \BN)\to F_Y(\Spec \BN)$ be the morphism induced by $f$. Then define for all $s\in F_X$ and $t=F(f)_*(s)\in F_Y$:
 \[N_t=\min\{h_Y(t')|t'\in F_Y^t(\Spec\BN), t'\notin F(f)_*(F_X(\Spec \BN))\},\]
 \[N_{s,t}=\min\{h_Y(t')|t'\in F(f)_*(F_X^s(\Spec \BN))\subset F_Y^t(\Spec\BN)\}\]
 and $N=\max\{N_t,N_{s,t}\}$.
 
 We have thus a finite partition
 \begin{align*}
  F_Y(\Spec\BN)=&\bigsqcup_{t\in F_Y} F_Y^t(\Spec\BN)\setminus F(f)_*(F_X(\Spec \BN)) \\
  &\sqcup\bigsqcup_{s\in F_X} F(f)_*(F_X^s(\Spec \BN)),
 \end{align*}
 where each partition subset contains at least one element with height less than $N$.
 
 Given some arbitrary $y\in V(k_v)$, we have to show that its fibre is $r$-cycle-split with a uniform bound $B$ on the maximum degree of witnesses so that we can conclude by Lemma~\ref{lem:birational}. The proof works by twice applying the logarithmic analogue of Hensel's lemma for log smooth morphisms from \cite{hensel} (see also \cite[\S5.2]{pseudo-split}). 
 
 By the above, we may find $b\in F_Y(\Spec\BN)$ in the same partition subset as $F(\wt y)$ with $h_Y(b)\leq N$. Write $b=\sum_{i\in I}b_i v_i$, where $(v_i)_{i\in I}$ are the cones in $F_Y$ corresponding to the irreducible components $(\CE_i)_{i\in I}$ of $\CE$.
 
 Let $(\pi_i)_{i\in I}$ be local equations for $(\CE_i)_{i\in I}$ in an affine neighbourhood $\Spec A$ of $(\tilde y \bmod \pi_v)$ in $\CY$.
 
 Let $\ov\varphi$ be the canonical morphism
 \[\ov \varphi:\CO_{k_v}\setminus 0\to (\CO_{k_v}\setminus 0)/(1+\pi_v\CO_{k_v})\cong k(v)^*\oplus \BN\to k(v)^*.\]
 
 The first application of logarithmic Hensel's lemma is to the diagram 
  \[\begin{tikzcd}
  \Spec(k(v))^\dagger \arrow[d]\arrow[r] & (\CY,\CE)\arrow[d]\\
  \Spec(\CO_{k_v})^\dagger \arrow[r] & \Spec(\CO_{k_v})^\mathrm{tr}
 \end{tikzcd}.\]
 Here, $\Spec(\CO_{k_v})^\mathrm{tr}$ denotes the trivial log structure with monoid $\CO_{k_v}^*$ and $\Spec(k(v))^\dagger$ denotes the standard log point with log structure $k(v)^*\oplus\BN$, the restriction of $\Spec(\CO_{k_v})^\dagger$. 
 
 On the level of monoids, the upper horizontal arrow is defined by
 \begin{eqnarray*}
  A^*\times \BN^I &\to& k(v)^*\oplus \BN,\\
  \alpha \in A^* &\mapsto& (\alpha(\tilde y \bmod \pi_v),0),\\
  1_i\in\BN^I &\mapsto& (\ov\varphi(\pi_i(\wt y)),b_i),
 \end{eqnarray*}
 where $1_i$ is the generator of the $i$-th factor. All other morphisms are the obvious ones.
 
 The point $y'\in Y(k_v)$ yielded by logarithmic Hensel's lemma has the same reduction as $y$ but satisfies \[F(\wt{y'})=b\] and
 \[\ov \varphi(\pi_i(\wt y))=\ov \varphi(\pi_i(\wt y')).\]
 
 This is the first half of the proof and works verbatim as in \cite[Proposition 6.1]{pseudo-split}.
 
 For the second half, the assumption of our proposition now states that $f^{-1}(y')$ contains a zero-cycle of degree $r$ which we write as $\sum_h n_h x_h'$. Here, $x_h'$ is a closed point defined over a finite extension $l_{w_h}/k_v$ with $[l_{w_h}:k_v]\leq B$. We are done with the proof, if we can lift each $(\wt{x_h'} \bmod \pi_{w_h})$ to an $l_{w_h}$-point $x_h\in f^{-1}(y)$.
 
 To do so, we only have to slightly alter diagram $(6.3)$ from the original proof in \cite{pseudo-split} and apply (for the second time) logarithmic Hensel, namely to
 \[\begin{tikzcd}
  \Spec(k(w_h))^\dagger \arrow[d]\arrow[r] & (\CX,\CD)\arrow[d]\\
  \Spec(\CO_{l_{w_h}})^\dagger \arrow[r] & (\CY,\CE)
 \end{tikzcd}.\]
 On schemes, the upper horizontal morphism is given by $(\wt{x_h'} \bmod \pi_{w_h})$ and the lower horizontal morphism is defined by $\wt y$ composed with $\Spec(\CO_{l_{w_h}})^\dagger \to \Spec(\CO_{k_v})^\dagger$.
 
Let $e_h$ be the ramification index of $l_{w_h}/k_v$. Then on fans \[\Spec(\CO_{l_{w_h}})^\dagger \to \Spec(\CO_{k_v})^\dagger\] is just $\Spec\BN\to\Spec\BN$ induced by multiplication with $e_h$ and hence \[F(\Spec(\CO_{l_{w_h}})^\dagger \to (\CY,\CE))=e_hF(\wt y).\]
 
 In an affine neighbourhood $\Spec(B)$ of $(\wt{x_h'} \bmod \pi_{w_h})$ in $\CX$, $(\CX,\CD)$ has a chart $\BN^J\to B$ given by sending the generator $1_j$ to a local equation $\omega_j$ of the irreducible component $\CD_j$. Let $u_j$ be the Kato subcone corresponding to $\CD_j$. Since $F(\wt y)$ and $b$ were chosen in the same partition subset and \[F(f)_*(\wt x_h')=F(\wt{y'})=b,\] there exists $a=\sum_j a_j u_j\in F_X^s(\Spec\BN)$ such that $F(\wt x_h')\in F_X^s(\Spec \BN)$ and $F(\CF)_*(a)=F(\wt y)$, so \[F(\CF)_*(e_ha)=F(\Spec(\CO_{l_{w_h}})^\dagger \to (\CY,\CE)).\] Then the log structure of $\Spec(k(w_h))^\dagger \to (\CX,\CD)$ should be defined by the morphism of monoids
 \[\BN^J\to k(w_h)^*\oplus \BN, 1_j\mapsto (\ov\varphi(\omega_j(\wt x_h')),e_ha_j).\]
 
 The proof that this defines a commuting diagram of log schemes works as in \cite[Proposition 6.1]{pseudo-split}.
\end{proof}

The next proposition \cite[Proposition 5.10 and Proposition 6.2]{pseudo-split} gives us a modification of $\CF$ (obtained by pulling back $N-1$ barycentric log blow-ups of the target) which will turn out to be optimal in the sense that it is all we need to check arithmetic $r$-cycle-surjectivity.
\begin{proposition}\label{prop:ultimate-mod}
 Let $N$ be a positive integer. There is a log smooth modification $\CF':(\CX',\CD')\to(\CY',\CE')$ of $\CF$ with $\CX'$ and $\CY'$ smooth, proper over $\CO_{k,S}$ and geometrically integral with the following property:
 
 Let $Y'$ be the generic fibre of $\CY'$ and $E'$ be the generic fibre of $\CE'$. For any $v\notin S$ and each point $y\in (Y\setminus E)(k_v)=(Y'\setminus E')(k_v)$ with $1\leq h_{\CY}(y)\leq N$, $h_{\CY'}(y)=1$ and its reduction in $\CY'$ is a smooth point of the reduction of $\CE'$. 
\end{proposition}

Now we can prove a sufficient criterion:
\begin{proposition}\label{prop:sufficient}
 Let $v\notin S$ and $\CF':(\CX',\CD')\to(\CY',\CE')$ a log smooth modification of $\CF$ as in Proposition\nobreakspace \ref {prop:ultimate-mod}. If $s^{0,r}_{f,\vartheta'}(v)=1$ for each generic point $\vartheta'$ of $D'$ (the generic fibre of $\CD'$), then $f\times_k{k_v}$ is $r$-cycle-surjective.
\end{proposition}
\begin{proof}
 By Chow's lemma, pick a rational embedding $\iota:\CX'_{k(v)}\dashrightarrow\BP^\nu_{\CY'_{k(v)}}$ and let $B=\Phi(\iota)$. Let $\CV':=\CX'\setminus\CE'$. It is enough to prove that the fibre over a point $y\in V'(k_v)=V(k_v)$ has a zero-cycle $Z$ of degree $r$ with $\maxdeg Z\leq B$. If the reduction of $y$ in $\CY$ is in $\CV$, we know that $\CF'^{-1}(\tilde y \bmod \pi_v)\cap\CU$ is non-empty smooth and geometrically integral (by assumption on the generic fibre), so $f^{-1}(y)$ has a zero-cycle of degree $1$ with maximum degree less than $B$ by the Lang-Weil estimates.
 
 Otherwise, assume that $\wt y$ intersects $\CE'$. By Proposition\nobreakspace \ref {prop:intersection-bound}, we can restrict ourselves to $y$ with $h_\CY(y)\leq N$. 
 
 Because of Proposition\nobreakspace \ref {prop:ultimate-mod}, $\wt y$ intersects transversally a codimension $1$ logarithmic stratum $\CZ$ of $(\CY',\CE')$. By Lemma\nobreakspace \ref {lem:irr-rep-strata} $\Irr^m_\CF$ is representable by a finite étale cover over logarithmic strata. Hence by assumption of $s^{0,r}_{f,\eta_\CZ}(v)=1$ and Corollary\nobreakspace \ref {cor:cycle-split}, the fibre $\CF'^{-1}(\tilde y \bmod \pi_v)$ is combinatorially $r$-cycle-split.
 
 The closure $\wt y$ of $y$ in $\CY'$ lies outside the Zariski closure of $E'_{\mathrm{sing}}$ (the singular locus of $E'$). Therefore $\CF'$ is integral outside the closure of $E'_\mathrm{sing}$ by \cite[Cor. 4.4(ii)]{fontaine-illusie}. Hence, the fibre product $\CX'_y:=(\CX',\CD')\times_{\CF',(\CY',\CE'),\wt y}(\Spec\CO_{k_v})^\dagger$, taken in the category of Zariski log schemes, is fine. Its underlying scheme agrees with the fibre product in schemes \cite[(1.6)]{fontaine-illusie}. Since $\wt y$ intersects $\CE'$ transversally, it follows that $\wt y:(\Spec\CO_{k_v})^\dagger\to (\CY',\CE')$ is a saturated morphism as in \cite{tsuji}. Hence by \cite[I.3.14]{tsuji}, $\CX'_y\to(\CX',\CD')$ is saturated and so is $\CX'_y$ \cite[II.2.12]{tsuji}. Thus $\CX'_y$ coincides with the fibre product taken in the category of fs log schemes.
 
 Log smoothness is stable under fs base change \cite[Proposition 12.3.24]{gabber-ramero}, so $\CX'_y$ is log regular, being log smooth over the log regular base $(\Spec\CO_{k_v})^\dagger$ \cite[Theorem 8.2]{kato}. It follows that $\CX'_y$ is Cohen-Macaulay and in particular normal \cite[Theorem 4.1]{kato}.
 
 That $\CF'^{-1}(y)=f^{-1}(y)$ is $r$-cycle-split with a witness $Z$ of $\maxdeg Z\leq B$ now follows from its reduction being combinatorially $r$-cycle-split and Lemma\nobreakspace \ref {lem:special-generic}.
\end{proof}

The main result Theorem\nobreakspace \ref {thm:main} reformulated for any $r\in\BN$ is now an easy corollary of Proposition\nobreakspace \ref {prop:non-surjectivity} and Proposition\nobreakspace \ref {prop:sufficient}.
\begin{theorem}\label{thm:mainr}
  Let $f:X\to Y$ be a dominant morphism between proper, smooth, geometrically integral varieties over a number field $k$ with geometrically integral generic fibre.
 
 Then $f$ is arithmetically $r$-cycle-surjective outside a finite set $S$, if and only if for each modification $f':X'\to Y'$ and for each codimension $1$  point $\vartheta'$ in $Y'$, the fibre $f'^{-1}(\vartheta')$ is combinatorially $r$-cycle-split.
\end{theorem}

\begin{remark}
The above result cannot be applied directly to \protect \MakeUppercase {C}onjecture\nobreakspace \ref {conj:artin}, which requires to prove that the exceptional set $S$ in Theorem\nobreakspace \ref {thm:main} is empty. We can nevertheless say the following.

In contrast to the case of Theorem\nobreakspace \ref {thm:lss}, the set $S$ for which we prove Theorem\nobreakspace \ref {thm:main} does not depend on Lang-Weil estimates but only on the existence of a sufficiently nice log smooth model of $f$ as stated in Section\nobreakspace \ref {sec:sufficient}. However, the existence of such models remains open. As far as zero-cycles are concerned, one may try to construct log smooth models by allowing alterations of $f$ instead of modifications and \cite{temkin} contains strong results in this direction. Unfortunately, even those models do not suffice since the creation of codimension $1$ logarithmic strata is not controlled.
\end{remark}

\begin{remark}
 Because the criterion of the preceding main theorem is stable under extensions of the ground field $k$, we could have also defined $r$-cycle-surjective to mean the existence of a zero-cycle of degree $r$ on each fibre over \emph{closed} points of $Y_{k_v}$ (instead of fibres over \emph{$k_v$-rational} points as in Definition\nobreakspace \ref {def:cycle-surjective}). The criterion of Theorem\nobreakspace \ref {thm:mainr} then shows that either definition leads to equivalent notions of arithmetic $r$-cycle-surjectivity (see the related observation by Liang \cite[Remark 6.5]{pseudo-split}).
 
 While using closed points is arguably the more natural definition, we prefer to keep Definition\nobreakspace \ref {def:cycle-surjective} in analogy with \cite{pseudo-split}.
\end{remark}

\begin{example}\label{ex:cycle-surjective}
 We give an example of a morphism for which one can show that it is arithmetically cycle-surjective but not arithmetically surjective.
 
 Let $A=\oplus_{i=1}^n k_i$ be a finite étale algebra over a number field $k$. Assume that $A$ is almost everywhere locally cycle-split but not pseudo-split (e.g.\ one of the algebras in Examples\nobreakspace \ref {ex:upgrade} and\nobreakspace  \ref {ex:non-upgrade}). Then one can define the multinorm torus $\mathrm{R}^1_{A/k}\BG_m$ through
 \[0\to\mathrm{R}^1_{A/k}\BG_m\to \mathrm{R}_{A/k}\BG_m\xrightarrow{N_{A/k}} \BG_m\to 0\]
 where the middle term maps to $\BG_m$ via the norm maps.
 
 The $1$-parameter family of torsors for $\mathrm{R}^1_{A/k}\BG_m$ given by 
 \[\mathrm{N}_{A/k}(x)=t\neq 0\]
 can be compactified to a proper, smooth, geometrically integral variety $X$ with a morphism $f$ to $\BP^1_k$.
 
 It is easy to see that for all $v\notin S$, all smooth fibres over $k_v$-points have a zero-cycle of degree $1$. Hence, $f$ is arithmetically cycle-surjective. On the other hand, since $A\otimes_k k_v$ is non-split for infinitely many $v$, it follows from \cite[Lemma 5.4]{smeets-loughran}, that $f$ is not arithmetically surjective.
\end{example}

\subsection*{Acknowledgements}
The author thanks M.~Bright, J.-L.~Colliot-Thélène, J.~Nicaise, A.~Skorobogatov and O.~Wittenberg and the anonymous referee for comments. This work was supported by the Engineering and Physical Sciences Research Council [EP/ L015234/1], the EPSRC Centre for Doctoral Training in Geometry and Number Theory (The London School of Geometry and Number Theory), University College London.

\bibliographystyle{alpha}
\newcommand{\etalchar}[1]{$^{#1}$}

\address{Department of Mathematics, South Kensington Campus\\
Imperial College London, LONDON, SW7 2AZ, UK\\
\email{d.gvirtz15@imperial.ac.uk}\\
\received{}}
\end{document}